\documentclass[12pt, reqno]{amsart}
\usepackage[cp1251]{inputenc}
\usepackage{amssymb, amsmath, amsthm}
\usepackage[colorlinks=true, citecolor=blue, urlcolor=blue]{hyperref}

\newtheorem{thrm}{Theorem}

\allowdisplaybreaks

\title[Additional first order equation for infinitesimal bendings of surfaces]
{Additional first order equation for infinitesimal bendings of smooth surfaces 
in the isothermal coordinates}
\author{Victor Alexandrov}
\address{Sobolev Institute of Mathematics, Koptyug ave., 4, Novosibirsk, 
630090, Novosibirsk, Russia and Department of Physics, Novosibirsk State 
University, Pirogov str., 2, Novosibirsk, 630090, Russia
\newline 
\indent 
\href{https://orcid.org/0000-0002-6622-8214}{\rm{ORCID iD: 0000-0002-6622-8214}}}
\email{alex@math.nsc.ru
\newline {}
\newline {}
\newline {$$*\quad *\quad *$$}
\newline {}
\textrm{\textbf{Journal reference for the article based on this preprint:}
{\sl Alexandrov V. A.} Additional first-order equation for infinitesimal bendings of
smooth surfaces in the isothermal coordinates. Siberian Math. J. {\bf 66}, no. 3, 
618--628 (2025). 
DOI: \href{https://doi.org/10.1134/S0037446625030024}{10.1134/S0037446625030024}.}}

\begin{document}

\begin{abstract}
The article contributes to the theory of infinitesimal bendings of smooth 
surfaces in Euclidean 3-space.
We derive a first-order linear differential equation, which previously 
did not appear in the literature and which is satisfied by any Darboux 
rotation field of a smooth surface.
We show that, for some surfaces, this additional equation is functionally 
independent of the three standard equations that the Darboux rotation 
field satisfies (and by which it is determined).
As a consequence of this additional equation, we prove the maximum 
principle for the components of the Darboux rotation field for a class 
of disk-homeomorphic surfaces containing not only surfaces of positive 
Gaussian curvature.
\par
\textit{UDC}: 514.7 	
\end{abstract}
\maketitle

\section*{\S\,1. Introduction}
The problems of  existence of isometric deformations and infinitesimal 
isometric deformations of surfaces in $\mathbb{R}^3$ arose in the works 
of Euler [1] and Gauss [2] simultaneously with the creation of differential 
geometry of surfaces.
The study of these problems has led to numerous profound results.
There are a vast number of books and articles that explore various aspects 
of these problems.
We mention only the fundamental survey article [3], the list of references in 
which contains 291 titles, and the book [4], the last one known to us, where 
the presentation of the theory of bending of surfaces starts from the very 
beginning, is conducted in all details, and contains all the issues necessary 
to understand this article.
Starting from them, the reader will be able to become familiar with the 
history of studying the theory of bending of surfaces, and with the main 
results of this theory.

The main results of the article are contained in Theorems~1--3.
In Theorem~1, we obtain a new first-order linear differential equation (11) 
not previously encountered in the literature, which the Darboux rotation field
must satisfy.
Examples~1--3 show that for some surfaces this new equation turns into 
the identity $0=0$, while for other surfaces (which can have both positive 
or negative curvature) this is a first-order equation that is linearly 
independent of the three standard first-order equations (8)--(10), which  
the Darboux rotation field satisfies.  
In Theorem~2, as a consequence of (11) we prove that under certain conditions 
the third component $y_3$ of a Darboux rotation field $y$ satisfies two 
homogeneous second-order linear partial differential equations (18) and (19).
We speak of the third component $y_3$ solely for the sake of definiteness, 
since the other two components of $y$ satisfy similar equations.
The existence of two second-order equations (18) and (19) for $y_3$ is a new 
effect that had not previously appeared in the theory of infinitesimal bendings.
The classical result is that the third component $z_3$ of an infinitesimal 
bending $z$ satisfies a single homogeneous second-order linear partial 
differential equation (see, e.g., [5; 6, Ch.~IV, \S\,2]).
Examples 4--6 show that a variety of situations is possible for equations (18)
and (19). 
Namely, in some cases they both degenerate (i.e., are not second-order 
equations); in some other cases, only one of them degenerates; and in some 
other cases, both (18) and (19) are second-order equations that are, moreover, 
functionally independent.
As a consequence of (18), in Theorem~3 we derive a maximum 
principle for $y_3$ for a certain class of disk-homeomorphic surfaces 
containing not only surfaces of positive Gaussian curvature.
For surfaces of positive Gaussian curvature this maximum principle cannot 
claim to be new because it is immediate from [7, Corollary of Theorem~2].

\section*{\S 2. Darboux rotation field}
 
Let
$\{S_t\}_{t\in(-1,1)}$ 
be a smooth family of smooth surfaces in 
$\mathbb{R}^3$.  
In this article, the word ``smooth'' has the usual meaning for classical
differential geometry, i.e., ``differentiable as many times as necessary 
in our reasoning.''
We always assume $S_t$ to be a connected surface, but unless explicitly 
stated, we do not assume other constraints such as compactness or no boundary.
By definition, put 
$S=S_0$.
The family 
$\{S_t\}_{t\in(-1,1)}$ 
is called an {\it isometric deformation} of
$S$,  
if for every 
$t\in(-1,1)$  
there is a smooth one-to-one map 
$f_t:S\to S_t$, . 
which preserves the length of every curve. 
Analytically, the last condition can be expressed as follows: 
for every parameterization
$x: U\subset\mathbb{R}^2 \to S\subset\mathbb{R}^3$ 
of
$S$
and the corresponding parameterization
$x_t\stackrel{\textrm{def}}{=}
f_t\circ x: U\subset\mathbb{R}^2 \to S_t\subset\mathbb{R}^3$ 
of
$S_t$ 
the first fundamental forms of  
$S$ 
and
$S_t$ 
coincide with each other, i.e., for all
$(u,v)\in U$ 
and every
$t\in(-1,1)$ 
we have
\begin{equation*}
\begin{cases}
((x_t)_u, (x_t)_u) = (x_u, x_u), \\ 
((x_t)_u, (x_t)_v) = (x_u, x_v), \\ 
((x_t)_v, (x_t)_v) = (x_v, x_v). 
\end{cases} \eqno(1)
\end{equation*}
Here are the subscripts 
$u$ 
and
$v$  
denotes differentiation with respect to the corresponding variable 
(the derivatives are calculated at the point 
$(u,v)\in U$); 
the subscript 
$t$ 
has a different meaning (it points to a certain surface $S_t$ from the family  
$\{S_t\}_{t\in(-1,1)}$);
finally, 
$(a,b)$ 
stands for the dot product of the vectors
$a,b\in\mathbb{R}^3$.
Relations (1) are nonlinear equations with respect to 
$x_t$. 
One of the classical approaches to solving them is to study their 
linearization, which can be obtained as follows.
We differentiate the equalities (1) by
$t$  
at 
$t=0$ 
and put
$$
z(u,v)\stackrel{\textrm{def}}{=}\frac{dx_t}{dt}\biggr|_{t=0}(u,v).
$$
As a result, we obtain the linear equations
$$
(z_u, x_u)=0, \quad
(z_u, x_v) + (z_v, x_u)=0,  \quad 
(z_v, x_v)=0  \eqno(2)
$$
with respect to the vector-valued function  
$z=z(u,v)$, 
the geometric meaning of which is well known;
namely, $z$ maps each point of
$S$ 
into the velocity vector of the point at the initial moment 
of the isometric deformation
$\{S_t\}_{t\in(-1,1)}$.
Any solution
$z$
to system (2), even if it is not generated by an isometric deformation 
$\{S_t\}_{t\in(-1,1)}$
of 
$S=S_0$, 
is called 
a {\it field of infinitesimal bending} of  
$S$
(or an {\it infinitesimal bending} of  
$S$).

An isometric deformation
$\{S_t\}_{t\in(-1,1)}$  
is {\it trivial}, if there is a smooth family 
$\{P_t\}_{t\in(-1,1)}$ 
of isometries
$P_t:\mathbb{R}^3\to \mathbb{R}^3$  
such that 
$S_t=P_t(S)$
for every
$t\in(-1,1)$. 
Otherwise, the isometric deformation 
$\{S_t\}_{t\in(-1,1)}$ 
is {\it nontrivial}.
A surface is {\it rigid}, if there is a nontrivial isometric deformation of it.
Otherwise, the surface is {\it nonrigid}.
An infinitesimal bending
$z$
of
$S$ 
is {\it trivial} if it can be generated by a trivial isometric deformation
$\{S_t\}_{t\in(-1,1)}$.
Otherwise, the infinitesimal bending is {\it nontrivial}.
A classical result of the theory of bendings of surfaces reads that if there 
is no nontrivial infinitesimal bending 
$z$ 
of 
$S_0$ 
then there is no nontrivial isometric deformation 
$\{S_t\}_{t\in(-1,1)}$ 
of 
$S_0$ 
which is analytic with respect to 
$t$.
The problem of the existence of a compact nonrigid surface without 
boundary is well known to specialists [8], but it remains open to this day.
On the other hand, it is known that for any natural  
$m\geq 1$  
there is a sphere-homeomorphic surface that has exactly 
$m$ 
linearly independent fields of nontrivial infinitesimal bending.
This statement was proved in [9] for $C^{\infty}$ surfaces and
in [10] for real-analytic surfaces.
 
Analytically it is quite difficult to recognize whether a given solution 
$z$
to (2) is a trivial infinitesimal bending or nontrivial. 
This is one of the reasons that in the theory of bendings of surfaces
it is considered standard to replace the vector-valued function 
$z=z(u,v)$  
by a new unknown vector-valued function
$y=y(u,v)$ 
using the formula 
$dz=y\times dx$ 
or, which is the same, by means of the formulas
\begin{gather*}
z_u=y\times x_u, 
                                                                \tag3
 \\ 
z_v=y\times x_v. 
                                                                \tag4
\end{gather*} 
Here 
$\times$ 
stands for the cross product. 
Differentiating (3) by
$v$, 
and (4) by 
$u$
and subtracting the second relation from the first, we obtain the system 
of differential equations
$$
y_u\times x_v =y_v\times x_u, \eqno(5)
$$
or, which is the same,
\begin{equation*}
\begin{cases}
x_{3v}y_{2u}-x_{2v}y_{3u}=x_{3u}y_{2v}-x_{2u}y_{3v},   \\ 
x_{1v}y_{3u}-x_{3v}y_{1u}=x_{1u}y_{3v}-x_{3u}y_{1v},  \\ 
x_{2v}y_{1u}-x_{1v}y_{2u}=x_{2u}y_{1v}-x_{1u}y_{2v}.  
\end{cases}\eqno(6)
\end{equation*}
Every solution 
$y$ 
to (5) is called a {\it Darboux rotation field} of
$S$ 
or of the infinitesimal bending 
$z$.
The name is explained by the fact that if we treat  
$z$ 
in (3)and (4) as an infinitesimal isometric motion of the tangent plane of
$S$, 
then
$y$ 
is the instantaneous angular velocity of the motion, i.e., 
$y$
is responsible for rotation (see [11]). 
This kinematic interpretation of 
$y$ 
makes it possible to prove that the infinitesimal bending
$z$ 
is trivial if and only if the corresponding Darboux rotation field 
$y$ 
is constant (see [11]).
It is also known that the Darboux rotation field has a few nontrivial 
geometric properties. 
The interested reader can find them, e.g., in the books [4,\,12], 
the articles [7,\,13--15], and the literature listed there.
We do not use those properties, and therefore do not formulate them.

\section*{\S\,3. First-order equations for the Darboux rotation field in 
isothermal coordinates}

Throughout the rest of this article
$S$ 
is a smooth surface,
$x:U\subset\mathbb{R}^2\to S\subset\mathbb{R}^3$ 
is a parameterization of 
$S$ 
by means of {\it isothermal} coordinates
$(u,v)\in U$
(i.e., such coordinates 
$(u,v)\in U$
that 
$$
(x_u,x_u)=\lambda, \quad (x_u, x_v)=0, \quad (x_v,x_v)=\lambda,\eqno(7)
$$
where
$\lambda:U\to (0,+\infty)$ 
is a smooth function).
By definition, we put
$n={\lambda}^{-1}(x_u\times x_v)$, 
so that
$n$ 
is the unit normal vector to 
$S$.
The symbols
$h_{ij}$ 
always denote the coefficients of the second fundamental form of   
$S$; 
namely,
\begin{gather*}
h_{11}= (x_{uu},n)=-(x_u,n_u), \\
h_{12}=h_{21}=(x_{uv},n)=-(x_u,n_v)=-(x_v,n_u), \\
h_{22}=(x_{vv},n)=-(x_v,n_v).
\end{gather*}
 
The existence of isothermal coordinates (i.e., the existence
of a surface parameterization by means of isothermal coordinates) was 
first established by Gauss in 1825 for analytical surfaces [16].  
Many researchers then worked to simplify the proofs and to lower the 
requirements for the smoothness of the surface on which the existence 
of isothermal coordinates can be guaranteed [17].
In the 1950s, the use of existence theorems for the Beltrami equation [18] 
became standard, which allowed us to study generalized solutions and 
significantly reduce the requirements on the smoothness of surfaces.
But there are other approaches, e.g., based on the use of the Fourier 
transform [19].
Another example: in 1953--1963, Yu.~G.~Reshetnyak published a series of 
nine articles devoted to the construction of isothermal coordinates on 
two-dimensional manifolds of bounded curvature in the sense of 
A.~D.~Alexandrov and the application of these coordinates to the study 
of such manifolds.
An English translation of these articles was published only in 2023 
in the book [20].
Since in this article we have limited ourselves to the framework of classical 
differential geometry, the minimal smoothness of the surface is not 
important for us.
Therefore, we do not even  explicitly formulate the mentioned results 
precisely.
What is important for us is that isothermal coordinates exist and are 
determined nonlocally in the sense that they can be introduced on any 
open part of the surface homeomorphic to the disk.
In general, all other special coordinate systems traditionally used in the 
surface theory are local (we mean, first of all, the normal, polar, and 
semigeodesic coordinate systems on a surface, see, e.g., [21, \S\,3.6]).

\begin{thrm}\label{thrm1}
Suppose that $S$ 
is a smooth surface and
$x:U\subset\mathbb{R}^2\to S\subset\mathbb{R}^3$  
is a parameterization of $S$ by the isothermal coordinates
$(u,v)\in U$. 
Let
$y:U\to \mathbb{R}^3$ 
be the Darboux rotation field of
$S$, 
so that $(5)$ holds in
$U$.
Then
$y$ 
also satisfies the following relations:
\begin{gather*}
(y_u,n)=0,  
                                                                \tag8
\\ 
(y_v,n)=0, 
                                                                \tag9
\\ 
(y_u,x_u)+(y_v,x_v)=0, 
                                                                \tag{10}
\\ 
h_{21}(y_u,x_u) +h_{22}(y_u,x_v)-h_{11}(y_v,x_u) -h_{12}(y_v,x_v)=0.
                                                                \tag{11}
\end{gather*}
\end{thrm}

\begin{proof}
By taking the dot product of both sides of (5) with 
$x_u$, 
we get
$$
(y_u\times x_v,x_u) =(y_v\times x_u,x_u).
$$
Using properties of the mixed product and equations (7), we obtain
$-\lambda (y_u,n)=0$. 
Therefore,
$y$
satisfies (8). 
Similarly, taking the dot product of both sides of (5) with
$x_v$
we see that
$y$
satisfies (9). 
Finally, taking the dot product of both sides of (5) with
$x_u\times x_v$ 
and using the well-known formula
$(a\times b, c\times d) = (a,c)(b,d) - (a,d)(b,c)$,
we successively get
\begin{gather*}
(y_u\times x_v,x_u\times x_v) =(y_v\times x_u,x_u\times x_v), \\
(y_u, x_u)(x_v,x_v)-(y_u,x_v)(x_v,x_u) =
(y_v,x_u)(x_u,x_v)-(y_v,x_v)(x_u,x_u), \\
\lambda (y_u,x_u)=-\lambda (y_v,x_v).
\end{gather*}
Hence
$y$ 
satisfies (10).
Thus, (8)--(10) are consequences of equations (5). 
The converse is also true.
Indeed, as we have just shown, (8)--(10) mean that the vectors 
$y_u\times x_v$ 
and 
$y_v\times x_u$ 
have the same coordinates in the orthogonal basis
$x_u$, $x_v$, $x_u\times x_v$. 
Hence,
$y_u\times x_v = y_v\times x_u$.
Thus, (5) are consequences of (8)--(10).

Differentiate (8) by 
$v$
and (9) by 
$u$,
and then subtract the second expression from the first.
This yields
$(y_u, n_v)-(y_v,n_u)=0$. 
Inserting here the expressions 
$n_u=-{\lambda}^{-1}[h_{11}x_u+h_{12}x_v]$ 
and 
$n_v=-{\lambda}^{-1}[h_{21}x_u+h_{22}x_v]$
known from differential geometry, we obtain
$$
-{\lambda}^{-1}[h_{21}(y_u,x_u) + h_{22}(y_u,x_v) 
- h_{11}(y_v,x_u) - h_{12}(y_v,x_v)]=0. 
$$ 
Therefore, (11) holds.
\end{proof}
  
It follows from the proof of Theorem 1 that the system of equations (8)--(10)
is equivalent to (6). 
This fact is well known.
The reader can find it, e.g., in [7,\,13--15] and the literature therein.
The situation is different with equation (11).
As far as the author knows, it has not been found before in the theory of 
infinitesimal bendings of surfaces. 
In a sense, finding such an ``additional'' equation (11) can be compared to 
finding the first integral of a dynamical system. 
This analogy explains the author's intention to explore what new 
consequences follow from the ``extended system of equations (8)--(11).'' 
In \S\,5, as one of these consequences, we derive a maximum principle for 
the third component $y_3$ of the Darboux rotation field $y$.

To better understand equations (5) and (8)--(11), let us consider some 
examples.

\textbf{Example 1}
($S$ 
is the plane 
$x_3=0$). 
Put
$U=\mathbb{R}^2$ and $x=(u,v,0)$. 

Then
$x_u=(1,0,0)$, $x_v=(0,1,0)$, $\lambda=1$, $n=(0,0,1)$, 
and
$h_{ij}=0$ 
for all 
$i,j=1,2$. 
Equations (5) take the form
$y_{3u}=0$, $y_{3v}=0$, 
and
$y_{1u}=-y_{2v}$;
and equations (8)--(10) take the form
$y_{3u}=0$, $y_{3v}=0$, 
and
$y_{1u}+y_{2v}=0$.
This is consistent with the conclusion of Theorem~1.
On the other hand, equation (11) becomes the identity 
$0=0$
and cannot provide any new information.

\textbf{Example 2}
($S$ 
is the sphere
$x_1^2+x_2^2+x_3^2=1$ 
with the point  
$(0,0,1)$
removed). 
We parameterize
$S$
using the stereographic projection, i.e., we put
$U=\mathbb{R}^2$ and $x=(2u,2v,u^2+v^2-1)/(u^2+v^2+1)$.

Then
$$
x_u=2(-u^2+v^2+1, -2uv, 2u)/(u^2+v^2+1)^2,
$$
$$
x_v=2(-2uv, u^2-v^2+1, 2v)/(u^2+v^2+1)^2,
$$
$$
\lambda=(x_u,x_u)=(x_v,x_v)=4/(u^2+v^2+1)^2, 
\quad 
(x_u,x_v)=0,
$$
$$
n={\lambda}^{-1}(x_u\times x_v)=-x, 
\quad 
h_{11}=-(x_u,n_u)=\lambda,
$$
$$
h_{12}=h_{21}=-(x_u,n_v)=-(x_v,n_u)=0, 
\quad 
h_{22}=-(x_v,n_v)=\lambda.
$$
After reduction by nonzero factors, (6) take the form
\begin{equation*}
\begin{cases}
2vy_{2u}-(u^2-v^2+1)y_{3u}=2uy_{2v}+2uvy_{3v}, \\ 
2vy_{1u}+2uvy_{3u}=2uy_{1v}+(-u^2+v^2+1)y_{3v}, \\ 
(u^2-v^2+1)y_{1u}+2uvy_{2u}=-2uvy_{1v}-(-u^2+v^2+1)y_{2v}. 
\end{cases}\eqno(12)
\end{equation*}
Similarly, after reduction by nonzero factors, equations (8)--(10) 
take the form
\begin{equation*}
\begin{cases}
2uy_{1u}+2vy_{2u}+(u^2+v^2-1)y_{3u}=0, \\ 
2uy_{1v}+2vy_{2v}+(u^2+v^2-1)y_{3v}=0, \\ 
(-u^2+v^2+1)y_{1u}-2uvy_{2u}+2uy_{3u}-2uvy_{1v}+(u^2-v^2+1)y_{2v}+2vy_{3v}=0. 
\end{cases}
\eqno(13)
\end{equation*}
Systems (12) and (13) look different.
But Theorem~1 states that they are equivalent to each other.
Finally, we write (11) as
$$
-2uvy_{1u}+(u^2-v^2+1)y_{2u}+2vy_{3u}-(-u^2+v^2+1)y_{1v}
+2uvy_{2v}-2uy_{3v}=0. \eqno(14)
$$
Let us write (13) and (14) in the matrix form
$AY=0\in\mathbb{R}^4$, 
where
$$
A=\begin{pmatrix} 
2u & 2v & u^2+v^2-1 & 0 & 0 & 0 \\ 
0 & 0 & 0 & 2u & 2v & u^2+v^2-1 \\
-u^2+v^2+1 & -2uv & 2u & -2uv & u^2-v^2+1 & 2v \\
-2uv & u^2-v^2+1 & 2v & -(-u^2+v^2+1) & 2uv & -2u
\end{pmatrix},
$$
and the column vector
$Y$
is the transpose of the row vector
$(y_{1u},y_{2u},y_{3u}, y_{1v},y_{2v},y_{3v})$. 
Denote by 
$A_{ij}$ 
the 
$4\times 4$-matrix
obtained from 
$A$ 
by removing the
$i$th
and
$j$th
columns.
A direct calculation gives
$$
\det A_{36}=4(u^2+v^2)(u^2+v^2+1)^2,
$$
$$
\det A_{25}=(u^2+v^2+1)^2(u^2+(v-1)^2)(u^2+(v+1)^2).
$$
Hence, for every 
$(u,v)\in\mathbb{R}^2$ 
there is a
$4\times 4$-minor of  
$A$ 
with non-zero determinant.
Thus, the rank of  
$A$
is always equal to four, i.e., for every 
$(u,v)\in\mathbb{R}^2$ 
equation (14) is not a linear combination of equations (13).

Therefore we see that in the case of a sphere, it makes sense to consider the
four equations (8)--(11) rather than the three equations (6) (or, equivalently, 
than the three equations (8)--(10)). 
 
\textbf{Example 3.}
As is known [22, \S\,16.2], a helicoid (more precisely, a certain region on it) 
can be continuously deformed into a catenoid (more precisely, into the 
corresponding region on it) using the following transformation:
$$
x_t(u,v)=(\cos t \sin u \sinh v + \sin t \cos u \cosh v, 
-\cos t \cos u \sinh v + \sin t \sin u \cosh v, 
u\cos t + v \sin t ).
$$
Here
$u\in (0, 2\pi]$, $v\in \mathbb{R}$, $t=0$ 
corresponds to a helicoid, and
$t=\pi/2$
corresponds to a catenoid.
Moreover, for every 
$t$,
the surface 
$S_t$
determined by the above parameterization $x_t$, 
is minimal and is isometric to the helicoid 
$x_0$.
Obviously,
\begin{gather*}
(x_t)_u = (\cos t \cos u \sinh v - \sin t \sin u \cosh v, 
\cos t \sin u \sinh v + \sin t \cos u \cosh v, \cos t), 
\\
(x_t)_v = (\cos t \sin u \cosh v + \sin t \cos u \sinh v, 
-\cos t \cos u \cosh v + \sin t \sin u \sinh v, \sin t ),\\
\lambda_t = ((x_t)_u,(x_t)_u)=((x_t)_v,(x_t)_v)=\cosh^2 v, 
  \quad ((x_t)_u,(x_t)_v)=0,\\ 
n_t = (\cosh v)^{-1}(\cos u, \sin u, -\sinh v),\\
(x_t)_{uu} = (-\cos t \sin u \sinh v - \sin t \cos u \cosh v, 
\cos t \cos u \sinh v - \sin t \sin u \cosh v, 0),\\
(x_t)_{uv} = (\cos t \cos u \cosh v - \sin t \sin u \sinh v, 
\cos t \sin u \cosh v + \sin t \cos u \sinh v, 0),\\
(x_t)_{vv} = (\cos t \sin u \sinh v + \sin t \cos u \cosh v, 
-\cos t \cos u \sinh v + \sin t \sin u \cosh v, 0),\\
(h_t)_{11} = ((x_t)_{uu},n_t)=-\sin t, \quad 
(h_t)_{12}=(h_t)_{21}=((x_t)_{uv},n_t)=\cos t, 
\\
(h_t)_{22}=((x_t)_{vv},n_t)=\sin t.
\end{gather*}

Recall that in the formulas just given, the previously adopted agreement 
is still in effect, according to which the lower index 
$t$
does not mean differentiation by
$t$,
but points to the surface 
$S_t$.  
Now we see that, for 
$S_t$, 
equations (6) take the form
\begin{equation*}
\begin{cases}
(\sin t) y_{2u}-(-\cos t \cos u \cosh v + \sin t \sin u \sinh v)y_{3u}\\
\hskip15mm
=(\cos t) y_{2v}-(\cos t \sin u \sinh v + \sin t \cos u \cosh v)y_{3v},\\
(\cos t \sin u \cosh v + \sin t \cos u \sinh v)y_{3u}-(\sin t)y_{1u}
\\
\hskip15mm
=(\cos t \cos u \sinh v - \sin t \sin u \cosh v)y_{3v}-(\cos t)y_{1v},\\ 
(-\cos t \cos u \cosh v + \sin t \sin u \sinh v)y_{1u}
-(\cos t \sin u \cosh v + \sin t \cos u \sinh v)y_{2u}
\\
\hskip8mm
=(\cos t \sin u \sinh v + \sin t \cos u \cosh v)y_{1v}
-(\cos t \cos u \sinh v - \sin t \sin u \cosh v)y_{2v}.
\end{cases}
							\eqno(15)
\end{equation*}

Similarly, after reduction by nonzero factors, equations (8)--(10)
take the form
\begin{equation*}
\begin{cases}
(\cos u) y_{1u} +(\sin u) y_{2u}- (\sinh v) y_{3u}=0,  \\ 
(\cos u) y_{1v} +(\sin u) y_{2v}- (\sinh v) y_{3v}=0,  \\ 
(\cos t \cos u \sinh v - \sin t \sin u \cosh v) y_{1u} +
(\cos t \sin u \sinh v + \sin t \cos u \cosh v) y_{2u}
  \\
\hskip20mm
+(\cos t) y_{3u}+ (\cos t \sin u \cosh v + \sin t \cos u \sinh v)y_{1v}
  \\ 
\hskip20mm
+(-\cos t \cos u \cosh v + \sin t \sin u \sinh v) y_{2v}+(\sin t) {y_3v}=0. 
\end{cases}
							\eqno(16)
\end{equation*}
 
But Theorem~1 states that systems (15) and (16) are equivalent to each other 
in the sense that for any fixed
$t$, $u$, 
and
$v$
each equation of system~(16) is a consequence of the equations of 
system~(15), and vice versa, each equation of system~(15) is a 
consequence of the equations of system~(16).

Finally, we write equation (11) for the surface
$x_t$
as
\begin{multline*}
(h_t)_{21}(y_{1u}(x_t)_{1u}+y_{2u}(x_t)_{2u}+y_{3u}(x_t)_{3u}) +
(h_t)_{22}(y_{1u}(x_t)_{1v}
\\
+y_{2u}(x_t)_{2v}+y_{3u}(x_t)_{3v}) 
-(h_t)_{11}(y_{1v}(x_t)_{1u}+y_{2v}(x_t)_{2u}
\\
+y_{3v}(x_t)_{3u}) -
(h_t)_{12}(y_{1v}(x_t)_{1v}+y_{2v}(x_t)_{2v}+y_{3v}(x_t)_{3v}) =0. 
\end{multline*}
Inserting here the above calculated components of the vectors  
$(x_t)_u$
and 
$(x_t)_v$,
as well as the coefficients 
$(h_t)_{ij}$, $i,j=1,2$,
of the second fundamental form of the surface 
$x_t$,
after collecting similar terms, we get
$$
(\cos u\sinh v)y_{1u} + (\sin u\sinh v)y_{2u}+y_{3u} -(\sin u\cosh v)y_{1v} 
+ (\cos u\cosh v)y_{2v}- 0\cdot y_{3v} =0. \eqno(17)
$$
Let us write equations (16)--(17) in the matrix form 
$BY=0\in\mathbb{R}^4$, 
where
\begin{equation*}
B=
\begin{pmatrix}
\cos u        & \sin u        & -\sinh v & 0              & 0             
   & 0 \\ 
0             & 0             & 0        & \cos u         & \sin u        
   & -\sinh v \\
\omega_1      & \omega_2      & \cos t   & \omega_3       & \omega_4      
   & \sin t \\
\cos u\sinh v & \sin u\sinh v & 1        & -\sin u\cosh v & \cos u\cosh v 
   & 0
\end{pmatrix},
\end{equation*} 
$$
\omega_1 = \cos t \cos u \sinh v - \sin t \sin u \cosh v,
\quad 
\omega_2 = \cos t \sin u \sinh v + \sin t \cos u \cosh v,
$$
$$
\omega_3 = \cos t \sin u \cosh v + \sin t \cos u \sinh v,
\quad 
\omega_4 = -\cos t \cos u \cosh v + \sin t \sin u \sinh v,
$$
and the column vector
$Y$
is the transpose of the row vector 
$(y_{1u},y_{2u},y_{3u}, y_{1v},y_{2v},y_{3v})$. 
Denote by 
$B_{36}$ 
the 
$4\times 4$-matrix, 
obtained from
$B$ 
by removing the $3$rd and $6$th columns. 
Direct calculation yields
$\det B_{36}=-(\cosh v)^2\sin t$.  
Hence, for all 
$u\in (0, 2\pi]$, $v\in \mathbb{R}$, 
and
$t\in (0,\pi)$
the rank of matrix 
$B$  
is equal to four, i.e., equation (17) is not a linear combination of 
equations (16) for the specified values of the parameters $u$, $v$, and $t$.
 
Thus, in the case of the surface 
$S_t$
we see the same thing that we have already seen in the case of the sphere:
it makes sense to consider the four equations (8)--(11) rather than the three 
equations (6) (or, equivalently, than the three equations (8)--(10)).

\section*{\S\,4. Second-order equations for components of the Darboux 
rotation field in isothermal coordinates}

One of the ways to prove that a compact convex surface without boundary 
is rigid bases on the maximum principle for one of the components of 
the field of infinitesimal bendings
$z$ 
of the surface (see, e.g., [4--6]).
However, to do this, it is necessary to reduce the system of equations (2) 
to a single second-order equation and make sure that the latter is elliptic.
Note that it is not sufficient to check that system (2) is elliptic, since 
the maximum principle does not always hold for solutions of elliptic systems
(see, e.g., [23; 24, Ch.~2, \S\,4]).

In \S\,4, we derive from system (8)--(11) two second-order linear 
partial differential equations, each of which is satisfied by
$y_3$, 
the third component of the Darboux rotation field
$y$.
The main difference from the classical theory of infinitesimal bendings of 
surfaces is that for the third component 
$z_3$
of the field of infinitesimal bendings
$z$
it is possible to write only one second-order linear differential equation 
(see, e.g., [4,\,6]).

\begin{thrm}\label{thrm2}
Suppose that $S$ is a smooth surface, 
$x=(x_1,x_2,x_3):U\subset\mathbb{R}^2\to S\subset\mathbb{R}^3$ 
is a parameterization of
$S$
by the isothermal coordinates  
$(u,v)\in U$, 
and
$\lambda=(x_u,x_u)$.
Let
$n=(n_1,n_2,n_3)={\lambda}^{-1}(x_u\times x_v)$ 
be a unit normal vector to
$S$, 
$h_{ij}$
be the coefficients of the second fundamental form of  
$S$,
and 
$y=(y_1,y_2,y_3):U\to \mathbb{R}^3$ 
be a Daurboux rotation field of
$S$.
If none of the three functions 
$n_1$, $n_2$,
and
$h_{11}x_{3v}^2-2h_{12}x_{3u}x_{3v}+h_{22}x_{3u}^2$
vanishes at any point in the domain
$U$,
then
$y_3$ 
satisfies the following two linear partial differential equations of order 
at most two:
\begin{gather*}
h_{22} y_{3uu} + h_{11} y_{3vv} + {\rho}_1y_{3u} + {\rho}_2y_{3v} = 0, 
                                                                   \tag{18}
\\ 
h_{12}y_{3uv} + {\rho}_3y_{3u} + {\rho}_4y_{3v} = 0
                                                                \tag{19}
\end{gather*}
everywhere in 
$U$. 
Here
${\rho}_i$, 
with
$i=1,\dots, 4$, 
are some smooth functions, possibly depending on 
$x=x(u,v)$
and its derivatives, but independent of
$y=(y_1,y_2,y_3)$
and its derivatives.
\end{thrm}

\begin{proof}
The following formulas are required below; they can be easily verified
by direct calculation:
\begin{equation*}
\begin{matrix}
\lambda n=x_u\times x_v,   
 & \text{i.e.,} & \lambda n_1=x_{2u}x_{3v}-x_{2v}x_{3u}, \\ 
 & & \lambda n_2=x_{1u}x_{3v}-x_{1v}x_{3u}, \\ 
 & & \lambda n_3=x_{1u}x_{2v}-x_{1v}x_{2u}; \\ 
x_u=x_v\times n,   
 & \text{i.e.,} & x_{1u}=n_3x_{2v}-n_2x_{3v}, \\ 
 & & x_{2u}=n_1x_{3v}-n_3x_{1v}, \\ 
 & & x_{3u}=n_2x_{1v}-n_1x_{2v}; \\ 
x_v=n\times x_u, 
 & \text{i.e.,} & x_{1v}=n_2x_{3u}-n_3x_{2u}, \\ 
 & & x_{2v}=n_3x_{1u}-n_1x_{3u}, \\ 
 & & x_{3v}=n_1x_{2u}-n_2x_{1u}. 
 \end{matrix}\eqno(20)
\end{equation*}  

Let us write equations (8)--(11) in the matrix form 
$DX=Z$, 
where
\begin{equation*}
D=\begin{pmatrix}
n_1                       & n_2                       
      & 0                            & 0 \\ 
0                         & 0                         
      & n_1                          & n_2 \\
x_{1u}                    & x_{2u}                    
      & x_{1v}                       & x_{2v} \\
h_{12}x_{1u}+h_{22}x_{1v} & h_{12}x_{2u}+h_{22}x_{2v} 
      & -(h_{11}x_{1u}+h_{12}x_{1v}) & -(h_{11}x_{2u}+h_{12}x_{2v})
\end{pmatrix}, 
			\eqno(21)
\end{equation*}
\begin{equation*}
X=\begin{pmatrix}
y_{1u} \\
y_{2u} \\
y_{1v} \\
y_{2v}
\end{pmatrix}, 
\quad
Z=\begin{pmatrix}
-n_3y_{3u} \\
-n_3y_{3v} \\
-(x_{3u}y_{3u}+x_{3v}y_{3v}) \\
-(h_{12}x_{3u}+h_{22}x_{3v})y_{3u}+(h_{11}x_{3u}+h_{12}x_{3v})y_{3v}
\end{pmatrix}.\eqno(22)
\end{equation*}
Put 
$d=\det D$. 
Using the Laplace expansion along the last row and repeatedly
using (20), we obtain
$$
d = h_{11}x_{3v}^2-2h_{12}x_{3u}x_{3v}+h_{22}x_{3u}^2.
$$
So, the conditions of Theorem~2 imply
$d\neq 0$.
Applying Cramer's rule to the system
$DX=Z$, 
we obtain
$$
y_{1u}=d_1/d, \quad
y_{2u}=d_2/d, \quad
y_{1v}=d_3/d, \quad
y_{2v}=d_4/d. \eqno(23)
$$
Here
$d_i$
stands for the determinant of the matrix 
$D_i$, 
that results from 
$D$ 
in (21) by replacing the $i$th column by the column vector 
$Z$  
of (22). 
For every
$i=1,\dots, 4$, 
put 
$d_i =p_iy_{3u}+q_iy_{3v}$, 
where
$p_i$ 
and 
$q_i$ 
are independent of 
$y$ 
and its derivatives.
Applying to 
$D_i$ 
the Laplace expansion along the last row and repeatedly using (20), we obtain
\begin{equation*}
\begin{cases}
p_1=h_{11}x_{1v}x_{3v} - 2h_{12}x_{1v}x_{3u} + h_{22}x_{1u}x_{3u},    
\qquad 
    q_1= h_{11}n_2\omega, 
\\
p_2=h_{11}x_{2v}x_{3v} - 2h_{12}x_{2v}x_{3u} + h_{22} x_{2u}x_{3u},   
\qquad 
	q_2=-h_{11}n_1\omega, 
\\
p_3= -h_{22}n_2\omega, 
\qquad 
q_3=h_{11}x_{1v}x_{3v} - 2h_{12}x_{1u}x_{3v} + h_{22}x_{1u}x_{3u}, 
\\
p_4= h_{22}n_1\omega, 
\qquad 
q_4=h_{11}x_{2v}x_{3v} - 2h_{12}x_{2u}x_{3v} + h_{22}x_{2u}x_{3u},
\end{cases}
							\eqno(24)
\end{equation*}                                
where
$\omega= \lambda n_3^2 + x_{3u}^2 + x_{3v}^2$. 
Note that 
$\omega>0$ 
in
$U$. 
 
The first and third formulas in (23) can be written as
$y_{1u}d-d_1=0$ 
and
$y_{1v}d-d_3=0$. 
Differentiating the first relation by
$v$
and the last relation by 
$u$ 
and subtracting the second relation from the first one, we obtain
$y_{1u}d_v - y_{1v}d_u - d_{1v} + d_{3u} = 0$. 
In the resulting formula, we insert 
$d_i =p_iy_{3u}+q_iy_{3v}$ 
and perform differentiations.
After simplification, we get
\begin{multline*} 
p_3y_{3uu} + (q_3-p_1)y_{3uv} -q_1y_{3vv} 
+ [p_{3u} - p_{1v} + (p_1d_v - p_3d_u)/d]y_{3u} 
\\
+ [q_{3u} - q_{1v} + (q_1d_v - q_3d_u)/d]y_{3v}=0. \tag{25}
\end{multline*}
Next, we transform (25) by replacing 
$p_1$, $p_3$, $q_1$, 
and 
$q_3$
in the coefficients at the second derivatives of 
$y_3$
with the corresponding expressions of (24). 
After simplification and division by 
$(-n_2)$, 
we get
$$ 
\omega h_{22} y_{3uu} + 2\lambda h_{12}y_{3uv} + \omega h_{11} y_{3vv} +
r_1y_{3u} + r_2y_{3v} = 0, \eqno(26)
$$
where
$r_1=(p_{1v}-p_{3u})/n_2+(p_3d_u-p_1d_v)/(n_2d)$,
$r_2=(q_{1v}-q_{3u})/n_2+(q_3d_u-q_1d_v)/(n_2d)$.

Similarly, we write the second and fourth formulas in (23) as
$y_{2u}d-d_2=0$ 
and 
$y_{2v}d-d_4=0$. 
Differentiating the first relation by
$v$
and the last relation by 
$u$ 
and subtracting the second relation from the first, we obtain
$y_{2u}d_v - y_{2v}d_u - d_{2v} + d_{4u} = 0$.
In the resulting formula, we insert 
$d_i =p_iy_{3u}+q_iy_{3v}$ 
and perform differentiations.
After simplification, we get
\begin{multline*}
p_4y_{3uu} + (q_4-p_2)y_{3uv} - q_2y_{3vv}  
+ [p_{4u} - p_{2v} + (p_2d_v - p_4d_u)/d]y_{3u} 
\\
+ [q_{4u}-q_{2v} + (q_2d_v - q_4d_u)/d]y_{3v}=0. \tag{27}
\end{multline*}
Next, we transform (27) by replacing 
$p_2$, $p_4$, $q_2$, 
and
$q_4$ 
in the coefficients at the second derivatives of 
$y_3$
with the corresponding expressions of (24).
After simplification and division by 
$n_1$, 
we get 
$$
\omega h_{22} y_{3uu} - 2\lambda h_{12}y_{3uv} + \omega h_{11} y_{3vv} +
r_3y_{3u} + r_4y_{3v} = 0,\eqno(28)
$$
where
$r_3= (p_{4u} - p_{2v})/n_1 + (p_2d_v - p_4d_u)/(n_1d)$,
$r_4= (q_{4u} - q_{2v})/n_1 + (q_2d_v - q_4d_u)/(n_1d)$. 

Adding (26) and (28) and dividing the result by 
$2\omega$,
we obtain (18), where 
${\rho}_1=(r_1+r_3)/(2\omega)$ 
and
${\rho}_2=(r_2+r_4)/(2\omega)$.
By subtracting (28) from (26) and dividing the result by 
$4\lambda$,
we obtain (19), where
${\rho}_3=(r_1-r_3)/(4\lambda)$ 
and
${\rho}_4=(r_2-r_4)/(4\lambda)$.
\end{proof}

For the first and second components 
$y_1$
and
$y_2$ 
of the Darboux rotation field
$y$,
equations similar to (18) and (19) can be obtained in the same way, 
as equations (18) and (19) were obtained in the proof of Theorem~2.
We do not intend to use the equations for 
$y_1$
and
$y_2$ 
in this article, so we do not write them out here to avoid deviating from 
the main line of reasoning.

Let us see what equations (18) and (19) look like for the surfaces 
considered in Examples~1--3 of \S\,3.

\textbf{Example 4}
(a continuation of Example 1, in which the plane 
$x=(u,v,0)$
was considered).
In this case, we cannot even write equations (18) and (19).
Indeed, in this case 
$n_1$ 
is identically zero.
Therefore, Theorem~2 does not apply.

\textbf{Example 5}
(a continuation of Example 2, in which the sphere
$x_1^2+x_2^2+x_3^2=1$
parameterized by the stereographic projection was considered).
In this case, the line 
$n_1=0$ 
is given by 
$v=0$; 
the line
$n_2=0$ 
is given by 
$u=0$; 
and the equation
$$
h_{11}x_{3v}^2-2h_{12}x_{3u}x_{3v}+h_{22}x_{3u}^2=0
$$
takes the form
$$
16\lambda (u^2+v^2)/(u^2+v^2+1)^4=0
$$ 
and defines the single point
$u=v=0$.
Therefore, in each quadrant 
$K_i$, 
with
$i=1,\dots, 4$,
defined by the coordinate axes 
$u$, $v$, 
we can write (18) and (19).
In this case, (18) takes the form
$y_{3uu} + y_{3vv} + (\rho_1/\lambda)y_{3u} + (\rho_2/\lambda)y_{3v} = 0$,
i.e., it is an elliptic second-order linear partial differential equation.
So, according to the maximum principle, the maximum of 
$y_3$ 
on the closure of
$K_i$ 
is attained on the boundary of 
$K_i$.
Finally, we note that (19) takes the form 
${\rho}_3y_{3u} + {\rho}_4y_{3v} = 0$, 
and thus is not a second-order partial differential equation at all.

\textbf{Example 6}
(a continuation of Example 3, in which a continuous family of isometric 
surfaces
$S_t$
was considered containing the helicoid and the catenoid).
The line
$(n_t)_1=0$ 
is given by the equaition
$\cos u=0$ 
(thus
$u=\pi/2+\pi j$, $j\in{\mathbb Z}$);
the line
$(n_t)_2=0$ 
is given by the equation
$\sin u=0$ 
(thus
$u=\pi k$, $k\in{\mathbb Z}$);
the equation
$$
(h_t)_{11}(x_t)_{3v}^2-2(h_t)_{12}(x_t)_{3u}(x_t)_{3v}
+(h_t)_{22}(x_t)_{3u}^2=0
$$ 
takes the form
$-2\cos^2 t\sin t=0$ 
(thus it has no solutions  
$t\in (0,\pi/2)$).
Therefore, for each 
$t\in (0,\pi/2)$
in each subdomain of the domain
$U$
disjoint with all straight lines
$u=k\pi/2$, 
we can write equations (18) and (19).
Under these conditions, (18) takes the form
$$
(\sin t)y_{3uu} - (\sin t) y_{3vv} + \rho_1 y_{3u} + \rho_2 y_{3v} =0
$$
and (19) takes the form
$$
(\cos t)y_{3uv}+ {\rho}_3y_{3u} + {\rho}_4y_{3v} = 0,
$$
i.e., for 
$t\in(0,\pi/2)$,
both equations are second-order hyperbolic linear partial differential 
equations without common characteristic lines.

Examples 4--6 show that for some surfaces equations (18) and (19) degenerate, 
but for some others they are essentially different second-order equations.
The existence of two equations for 
$y_3$ 
(i.e., for the third component of the Darboux rotation field) is a new fact 
in the theory of infinitesimal bendings of smooth surfaces in 
$\mathbb{R}^3$.
The problem of using this fact is interesting, but it goes beyond the scope 
of this article.

\section*{\S\,5. The maximum principle for the components of the Darboux 
rotation field of surfaces with not necessarily positive curvature}

\begin{thrm}\label{thrm3}
Suppose that
$S$ 
is a smooth surface, 
$x=(x_1,x_2,x_3):U\subset\mathbb{R}^2\to S\subset\mathbb{R}^3$ 
is a parameterization of $S$ by the isothermal coordinates  
$(u,v)\in U$, 
and
$\lambda=(x_u,x_u)$.
Let 
$n=(n_1,n_2,n_3)={\lambda}^{-1}(x_u\times x_v)$ 
be a unit normal vector to
$S$, 
let 
$h_{ij}$ 
be the coefficients of the second fundamental form of  
$S$,
and let
$y=(y_1,y_2,y_3):U\to \mathbb{R}^3$ 
be a Daurboux rotation field of
$S$.
If 
\begin{equation*}
n_1\neq 0, \quad n_2\neq 0, \quad
h_{11}x_{3v}^2-2h_{12}x_{3u}x_{3v}+h_{22}x_{3u}^2\neq 0,\eqno(29)
\end{equation*}
\begin{equation*} 
h_{22}h_{11} > 0\eqno(30)
\end{equation*}
in 
$U$,
then the maximum of 
$y_3$ 
on the closure of 
$U$ 
is attained on the boundary of 
$U$. 
\end{thrm}

\begin{proof}
Owing to (29),
$y_3$ 
satisfies (18), i.e., 
$$
h_{22} y_{3uu} + h_{11} y_{3vv} + {\rho}_1y_{3u} + {\rho}_2y_{3v} = 0.
$$
By virtue of (30), the latter equation is elliptic. 
Therefore, the claim of Theorem 3 follows from the standard formulation 
of the maximum principle for elliptic equations (see, e.g., 
[24; 25, Part~2, \S\,2.2]). 
\end{proof}

It is clear that the third component 
$y_3$
of the Darboux rotation field
$y$ 
was considered in Theorem 3 solely for definiteness; statements similar to 
Theorem 3, of course, are valid for other components of the Darboux rotation 
field~
$y$.

Theorem 3 implies that if (30) holds on a parameterized surface
$x=(x_1,x_2,x_3):U\subset\mathbb{R}^2\to S\subset\mathbb{R}^3$,
then the maximal and minimal values of the function 
$y_3$
cannot be attained anywhere except on the boundary of the surface or 
on the lines defined by one of the equations
$$
n_1=0, \qquad n_2=0, \qquad 
h_{11}x_{3v}^2-2h_{12}x_{3u}x_{3v}+h_{22}x_{3u}^2=0.
$$ 
Thus, on some surfaces of not necessarily positive curvature (namely, 
those on which inequality (30) holds), we are able to construct lines
on which alone the maximal and minimal values of the function 
$y_3$
can be attained. 
Previously, such lines had not been encountered in the theory of 
infinitesimal bendings of smooth surfaces in $\mathbb{R}^3$.
From our point of view, they are similar to the nodal lines of standing 
waves (see, e.g., [26, Ch.~V, \S\,3]) and deserve a separate study.

However, we note that if the surface 
$S$
not only satisfies (30) (i.e., the inequality 
$h_{11}h_{22}>0$),
but also has everywhere it has positive Gaussian curvature (i.e., 
$h_{11}h_{22}-h^2_{12}>0$
on 
$S$), 
then the maximum principle for
$y_3$
holds without additional restrictions~(29).
This is immediate from Corollary to Theorem~2 in~[7], which states that 
the rotation diagram
$Y$ 
any nontrivial infinitesimal bending
$z$
of a surface of positive curvature at any internal point (i.e., a point that 
lies outside the boundary of the surface) does not have a local supporting 
plane.
Recall that the {\it rotation diagram}
$Y$
of an infinitesimal bending
$z$
is the set of endpoints of the vectors of the corresponding Darboux rotation 
field 
$y$
originated at some fixed point, e.g., at the origin.
The statement that if the Gaussian curvature of 
$S$
is positive, then at every regular point of 
$Y$
(i.e., at each point subject to $y_u\times y_v\neq 0$)
the rotation diagram 
$Y$
has strictly negative curvature (and, thus, 
$Y$ 
has no supporting plane at the point) was known to the 
classics (see, e.g., [11, \S\,5]).
Thus, the essence of Corollary to Theorem 2 in [7] is in the fact that if 
the Gaussian curvature of 
$S$
is positive, then there is no local supporting plane at every singular point of 
the rotation diagram
$Y$.
For this purpose, topological properties of Sto\"{\i}low-inner mappings
(i.e., open and light mappings; see [27]) are used in [7].

\subsection*{Funding}
The research was carried out within the State Task to the Sobolev Institute of 
Mathematics (Project FWNF--2022--0006).

\section*{References}
{\small 

\noindent{[1]} {\sl Euleri L.}
Fragmentum 97 // 
In: A. Speiser (ed.). Leonhardus Eulerus. Opera omnia, Ser. 1. 
Opera mathematica. V. 13. Commentationes geometricae. V. 4.
Lausannae: Auctoritate et impensis Societatis Scientiarum Naturalium 
Helveticae (1956). P. 437--440.

\noindent{[2]} {\sl Gauss K. F.}
General investigations of curved surfaces.
Hewlett, New York: Raven Press (1965).
[Translation from the Latin and German.] 

\noindent{[3]} {\sl Ivanova-Karatopraklieva I.} and {\sl Sabitov I. Kh.}
Surface deformation. I. 
J. Math. Sci., New York. 1994. V. 70, no. 2, 1685--1716.
[Translation from the Russian.]

\noindent{[4]} {\sl Klimentov~S.~B.}
An introduction to the bending theory. Two-dimensional surfaces 
in three-dimensional Euclidean space. [In Russian].
Rostov-on-Don: Publishing House of the Southern Federal University, 2014. 
[In~electronic form, the book is available at
\href{https://elibrary.ru/item.asp?id=24164017}{https://elibrary.ru/item.asp?id=24164017}.]

\noindent{[5]} {\sl Bojarskij B. V.} and {\sl  Efimov N. V.}
The maximum principle for infinitesimal bendings of piecewise regular
convex surfaces. [In Russian]. 
Uspekhi Mat. Nauk. 1959. V. 14, no. 6, 147--153. 

\noindent{[6]} {\sl Pogorelov A. V.}
Extrinsic geometry of convex surfaces. 
Providence, R.I.: American Mathematical Society (1973).
[Translation from the Russian.]

\noindent{[7]} {\sl Sabitov I. H.}
Local structure of Darboux surfaces.
Sov. Math., Dokl. 1965. V. 6, 804--807.
[Translation from the Russian.]

\noindent{[8]} {\sl Yau S.-T.} 
Problem section. In: S.-T. Yau (ed.). Seminar on differential geometry.
Ann. Math. Stud. 1982. V. 102, 669--706.

\noindent{[9]} {\sl Reshetnyak Yu. G.} 
Nonrigid surfaces of revolution. [In Russian]. 
Sib. Mat. Zh. 1962. V. 3, no. 4, 591--604.

\noindent{[10]} {\sl Trotsenko D. A.} 
Nonrigid analytic surfaces of revolution. 
Siberian Math. J. 1980. V. 21, no. 5, 718--724.
[Translation from the Russian.]

\noindent{[11]} {\sl Cohn-Vossen S. {\`E}.} 
Bending of surfaces in the large. [In Russian].
Uspekhi Mat. Nauk. 1936, issue 1, 33--76.

\noindent{[12]} {\sl Darboux G.} 
Le\,cons sur la th{\'e}orie g{\'e}n{\'e}rale des surfaces et les applications
g{\'e}om{\'e}triques du calcul infinit{\'e}simal. T. IV: 
D{\'e}formation infiniment petite et repr{\'e}sentation sph{\'e}rique.
Paris: Gauthier-Villars (1896).
 
\noindent{[13]} {\sl Rembs E.} 
Verbiegungen h\"oherer Ordnung und ebene Fl\"achenrinnen.
Math. Z. 1933. Bd. 36. S. 110--121. 

\noindent{[14]} {\sl Efimov N. V.} 
Qualitative problems of the theory of deformation of surfaces.
[In Russian].
Uspekhi Mat. Nauk. 1948. V. 3, issue 2, 47--158.

\noindent{[15]} {\sl Alexandrov V.} 
New manifestations of the Darboux's rotation and translation fields of 
a surface. New Zealand Journal of Mathematics. 2010. V. 40, 59--65. 

\noindent{[16]} {\sl Gauss C.F.}
On conformal representation. In: D.E. Smith (ed.).
A source book in mathematics. V. 2. P. 463--475.
New York: Dover Publications, 1959. 

\noindent{[17]} {\sl Chern S.-S.}
An elementary proof of the existence of isothermal parameters on a surface.
Proc. Am. Math. Soc. 1955. V. 6, 771--782. 

\noindent{[18]} {\sl Vekua I. N.}
Generalized analytic functions.
Oxford etc.: Pergamon Press (1962).
[Translation from the Russian.]

\noindent{[19]} {\sl Douady A.} 
Le th{\'e}or{\`e}me d'int{\'e}grabilit{\'e} des structures presque complexes 
(d’apr{\`e}s des notes de X. Buff) // In: T. Lei (ed.). 
The Mandelbrot set, theme and variations.
Lond. Math. Soc. Lect. Note Ser. 2000. V. 274. P. 307--324.
Cambridge: Cambridge University Press. 

\noindent{[20]} {\sl Fillastre F.} and {\sl Slutskiy D.} (eds.).
Reshetnyak’s theory of subharmonic metrics. 
Cham: Springer (2023). 

\noindent{[21]} {\sl Toponogov V. A.}
Differential geometry of curves and surfaces. 
Basel: Birkh{\"a}user (2005). 
[Translation from the Russian.]

\noindent{[22]} {\sl Gray A., Abbena E.} and {\sl Salamon S.}
Modern differential geometry of curves and surfaces with Mathematica. 3rd ed. 
Boca Raton: CRC (2006).

\noindent{[23]} {\sl Maz'ya V. G.} and {\sl Kresin G. I.}
On the maximum principle for strongly elliptic and parabolic second order 
systems with constant coefficients. Math. USSR-Sb. 1986. V. 53, no. 2, 457--479.
[Translation from the Russian.]

\noindent{[24]} {\sl Bitsadze A. V.}
Boundary value problems for second order elliptic equations. 
Amsterdam: North-Holland Publishing Company (1968). 
[Translation from the Russian.]

\noindent{[25]} {\sl Bers L., John F.} and {\sl Schechter M.}
Partial differential equations.
Providence, R.I.: American Mathematical Society (1979).

\noindent{[26]} {\sl Tikhonov A. N.} and {\sl Samarskii A. A.}
Equations of mathematical physics. Oxford etc.: Pergamon Press (1963).  
[Translation from the Russian.] 

\noindent{[27]} {\sl Sto\"{\i}low S.}
Le\c{c}ons sur les principes topologiques de la th{\'e}orie des fonctions analytiques.
Paris: Gauthier-Villars (1956). 

}

\end{document}